\documentclass{amsart}
\usepackage{amsfonts}

\newtheorem{theorem}{Theorem}
\theoremstyle{plain}

\newtheorem{corollary}{Corollary}

\newtheorem{definition}{Definition}
\newtheorem{example}{Example}

\newtheorem{lemma}{Lemma}

\newtheorem{proposition}{Proposition}
\newtheorem{remark}{Remark}

\numberwithin{equation}{section}
\input{tcilatex}

\begin{document}
\title[orthogonal polynomials]{A few remarks on orthogonal polynomials}
\author{Pawe\l\ J. Szab\l owski}
\address{Department of Mathematics and Information Sciences,\\
Warsaw University of Technology\\
ul. Koszykowa 75, 00-662 Warsaw, Poland}
\email{pawel.szablowski@gmail.com}
\date{January 2013}
\thanks{The author is grateful to unknown referees whose remarks helped to
improve the paper.\\
Email: pawel.szablowski@gmail.com}

\begin{abstract}
Knowing a sequence of moments of a given, infinitely supported, distribution
we obtain quickly: coefficients of the power series expansion of monic
polynomials $\left\{ p_{n}\right\} _{n\geq 0}$ that are orthogonal with
respect to this distribution, coefficients of expansion of $x^{n}$ in the
series of $p_{j},$ $j\leq n$, two sequences of coefficients of the 3-term
recurrence of the family of $\left\{ p_{n}\right\} _{n\geq 0}$, the so
called "linearization coefficients" i.e. coefficients of expansion of $%
p_{n}p_{m}$ in the series of $p_{j},$ $j\leq m+n.$\newline
Conversely, assuming knowledge of the two sequences of coefficients of the
3-term recurrence of a given family of orthogonal polynomials $\left\{
p_{n}\right\} _{n\geq 0},$ we express with their help: coefficients of the
power series expansion of $p_{n}$, coefficients of expansion of $x^{n}$ in
the series of $p_{j},$ $j\leq n,$ moments of the distribution that makes
polynomials $\left\{ p_{n}\right\} _{n\geq 0}$ orthogonal. \newline
Further having two different families of orthogonal polynomials $\left\{
p_{n}\right\} _{n\geq 0}$ and $\left\{ q_{n}\right\} _{n\geq 0}$ and knowing
for each of them sequences of the 3-term recurrences, we give sequence of
the so called "connection coefficients" between these two families of
polynomials. That is coefficients of the expansions of $p_{n}$ in the series
of $q_{j},$ $j\leq n.$\newline
We are able to do all this due to special approach in which we treat vector
of orthogonal polynomials $\left\{ p_{j}\left( x)\right) \right\} _{j=0}^{n}$
as a linear transformation of the vector $\left\{ x^{j}\right\} _{j=0}^{n}$
by some lower triangular $(n+1)\times (n+1)$ matrix $\mathbf{\Pi }_{n}.$
\end{abstract}

\maketitle

\section{Introduction and notation}

Let us first make some remarks concerning notation. $\alpha ,$ $\beta ,$ $%
\ldots $ will denote positive measures on the real line. We will assume that
all of these measures have infinite supports. In order to be able to
sometimes use probabilistic notation we will assume that all considered
measures are normalized. Integrals of an integrable function $f$ with
respect to the measure $\alpha $ will be denoted by either of the following
denotations: 
\begin{equation*}
\int f(x)d\alpha (x),~~\int fd\alpha ,\mathbb{E}f,~~\mathbb{E}f(Z),~~\mathbb{%
E}_{\alpha }f(Z),
\end{equation*}%
depending on the context and the need to specify details. In the above
formulae, $Z$ denotes random variable with the distribution $\alpha .$
Probability theory ensures that $Z$ always exist.

Matrices and vectors (always columns) will be generally denoted by the bold
type letters. The most important vector and matrix are $\mathbf{X}%
_{n}\allowbreak =\allowbreak (1,x,\ldots ,x^{n})^{T}$ ($T-$transposition)
and 
\begin{equation}
\mathbf{M}_{n}(\alpha )\allowbreak =\allowbreak \left[ m_{i+j}(\alpha )%
\right] _{j,i=0,\ldots ,n},  \label{matrix}
\end{equation}%
where $m_{n}(\alpha )\allowbreak =\allowbreak \int x^{n}d\alpha (x).$ In
other words $\mathbf{M}_{n}(\alpha )\allowbreak =\allowbreak \mathbb{E}%
_{\alpha }\mathbf{X}_{n}\mathbf{X}_{n}^{T}.$ Matrices of this form, that is
having the same elements on counter diagonals, are called Hankel matrices.

\begin{definition}
We will say that the moment problem is determinate if there exists only one
measure $\alpha $ that generates the moment sequence $\left\{ m_{n}\left(
\alpha \right) \right\} _{n\geq 0}.$ Otherwise we say that the moment
problem is indeterminate.
\end{definition}

\begin{remark}
There exist sufficient criteria allowing to check if the moment problem is
determinate or not. For example Carleman's criterion states that if $%
\sum_{n\geq 0}m_{2n}^{1/2n}<\infty ,$ then the moment problem is
indeterminate. Else, if $\int \exp \left( \left\vert x\right\vert \right)
d\alpha \left( x\right) <\infty ,$ then the moment problem is determinate.
\end{remark}

In the sequel we will assume generally that our moment problem is
determinate.

$\left( \mathbf{A}\right) _{j,k}$ will denote $(j,k)-$th entry of the matrix 
$\mathbf{A.}$

Infinite support assumption ensures that for every $n$ one can always find $%
n+1$ linearly independent vectors of the form $(1,x_{k},\ldots
,x_{k}^{n})^{T},$where $x_{k}\in \limfunc{supp}\alpha $. Besides we know
that if $\limfunc{supp}\alpha $ is infinite then matrices $\mathbf{M}%
_{n}(\alpha )$ are non-singular for every $n.$ Let us remark immediately
that the matrix $\mathbf{M}_{n}$ is the main submatrix of the matrix $%
\mathbf{M}_{n+1}.$ We also define sequence 
\begin{equation}
\Delta _{n}(\alpha )\allowbreak =\allowbreak \det \mathbf{M}_{n}(\alpha ),
\label{det}
\end{equation}%
$n\geq 1$, of determinants of matrices $\mathbf{M}_{n}(\alpha )$ and let us
also introduce vectors consisting of successive moments 
\begin{equation*}
\mathbf{m}_{n}^{T}(\alpha )\allowbreak =\allowbreak (1,\ldots ,m_{n}(\alpha
)).
\end{equation*}%
Vector $\mathbf{m}_{n}(\alpha )$ is the first column of the matrix $\mathbf{M%
}_{n}(\alpha ).$

In order to avoid repetition of assumption we will assume that matrices $%
\mathbf{M}_{n}(\alpha )$ exist for all $n\geq 0.$ In other words we assume
that all moments of the measure $\alpha $ exist. Obviously $(0,0)$ entry of
the matrix $\mathbf{M}_{n}$ is equal to $1.$

We know that given measure $\alpha ,$ such that all moments exist, one can
define the set of polynomials $\left\{ p_{n}(x,\alpha )\right\} _{n\geq -1}$
with $p_{-1}(x,\alpha )\allowbreak =\allowbreak 0,$ $p_{0}(x,\alpha
)\allowbreak =\allowbreak 1$, such that $p_{n}$ is of degree $n$ and
satisfying for $n+m\neq -2$ the following relationship:%
\begin{equation*}
\int p_{n}\left( x,\alpha \right) p_{m}(x,\alpha )d\alpha (x)\allowbreak
=\allowbreak \delta _{n,m},
\end{equation*}%
where $\delta _{n,m}$ denotes Kronecker's delta. Moreover if we declare that
all leading coefficients of the polynomials $p_{n}(x,\alpha )$ are positive
then the coefficients $\pi _{n,i}(\alpha )$ of the expansion%
\begin{equation}
p_{n}\left( x,\alpha \right) \allowbreak =\allowbreak \sum_{i=0}^{n}\pi
_{n,i}\left( \alpha \right) x^{i},  \label{p_o}
\end{equation}%
are defined uniquely by the measure $\alpha .$ According to our convention,
later we will drop dependence on $\alpha $ if the measure $\alpha $ is
clearly specified.

Let us define vectors $\mathbf{P}_{n}(x)\allowbreak =\allowbreak
(p_{0}(x),\ldots ,p_{n}(x))^{T}$ and the lower triangular matrix $\mathbf{%
\Pi }_{n}$ with entries $\pi _{i,j}$ if $i\geq j$ and $0$ otherwise. We
obviously have:%
\begin{equation}
\mathbf{P}_{n}(x)=\mathbf{\Pi }_{n}\mathbf{X}_{n}.  \label{_p}
\end{equation}%
To continue introduction of notation, let $\lambda _{n,i}(\alpha )$ denote
coefficients in the following expansions:%
\begin{equation}
x^{n}\allowbreak =\allowbreak \sum_{i=0}^{n}\lambda _{n,i}(\alpha
)p_{i}\left( x,\alpha \right) .  \label{inv_p}
\end{equation}%
Consequently let us introduce lower triangular matrices $\mathbf{\Lambda }%
_{n}$ with entries $\lambda _{i,j}$ if $i\geq j$ and $0$ otherwise.

We obviously have:%
\begin{equation}
\mathbf{X}_{n}\allowbreak =\allowbreak \mathbf{\Lambda }_{n}\mathbf{P}%
_{n}(x),~~\mathbf{\Pi }_{n}\mathbf{\Lambda }_{n}\allowbreak =\allowbreak 
\mathbf{\Lambda }_{n}\mathbf{\Pi }_{n}=\mathbf{I}_{n},  \label{_x}
\end{equation}%
where $\mathbf{I}_{n}$ denotes $(n+1)\times (n+1)$ identity matrix.

As polynomials $\left\{ p_{n}\right\} $ are orthonormal, there exist two
number sequences $\left\{ a_{n}\right\} ,$ $\left\{ b_{n}\right\} $ such
that polynomials $\left\{ p_{n}\right\} $ satisfy the following 3-term
recurrence:%
\begin{equation}
xp_{n}(x)\allowbreak =\allowbreak
a_{n+1}p_{n+1}(x)+b_{n}p_{n}(x)+a_{n}p_{n-1}(x),  \label{3tr}
\end{equation}%
with $a_{0}\allowbreak =\allowbreak 0$ and $n\geq 0.$ We know also that 
\begin{equation}
a_{n}\allowbreak =\allowbreak \frac{\pi _{n-1,n-1}}{\pi _{n,n}}%
,~~b_{n}\allowbreak =\allowbreak \int xp_{n}^{2}\left( x\right) d\alpha
\left( x\right) ,  \label{coeff}
\end{equation}%
consequently that $b_{0}\allowbreak =\allowbreak m_{1}.$ For details see
e.g. \cite{Akhizer65} or \cite{Nev86}. \ 

Combining (\ref{3tr}) and (\ref{p_o}) we get the following set of recursive
equations to be satisfied by the coefficients $\pi _{n,j}$. 
\begin{eqnarray}
a_{n+1}\pi _{n+1,0}+b_{n}\pi _{n,0}+a_{n}\pi _{n-1,0}\allowbreak
&=&\allowbreak 0,  \label{_f} \\
a_{n+1}\pi _{n+1,j}+b_{n}\pi _{n,j}+a_{n}\pi _{n-1,j}\allowbreak
&=&\allowbreak \pi _{n,j-1},  \label{_s}
\end{eqnarray}%
for $n\geq 0,$ $j\allowbreak =\allowbreak 1,\ldots ,n$, with $\pi
_{n,j}\allowbreak =\allowbreak 0$ for $j>n.$ Further combining (\ref{3tr})
and (\ref{inv_p}) we get the following set of equations to be satisfied by
the coefficients $\lambda _{n,i}.$ 
\begin{eqnarray}
\lambda _{n+1,n+1}\allowbreak &=&\allowbreak \lambda _{n,n}a_{n+1},
\label{lam1} \\
\lambda _{n+1,0}\allowbreak &=&\allowbreak \lambda _{n,0}b_{0}+\lambda
_{n,1}a_{1},  \label{lam2} \\
\lambda _{n+1,i} &=&\lambda _{n,i-1}a_{i}+\lambda _{n,i}b_{i}+\lambda
_{n,i+1}a_{i+1},  \label{lam3}
\end{eqnarray}%
with, $\lambda _{0,0}\allowbreak =\allowbreak 1,$ so $\lambda
_{n,n}\allowbreak =\allowbreak \prod_{j=1}^{n}a_{j},$ $n\geq 1.$

\begin{remark}
As it follows from formula (2.1.6) of \cite{IA} coefficients $\pi _{n,i}$
can be expressed as determinants of certain submatrices built of moment
matrix $\mathbf{M}_{n}.$ In particular denoting by $D_{n}^{(i,j)}$ the
determinant of a submatrix obtained by removing row number $i+1$ and column
number $j+1$ of the matrix $\mathbf{M}_{n}.$ We have $\pi _{n,i}\allowbreak
=\allowbreak (-1)^{n-i}D_{n}^{(i,n)}/\sqrt{\Delta _{n}\Delta _{n-1}}$ the so
called Heine representation of orthogonal polynomials (see formula (2.2.6)
in \cite{IA}).
\end{remark}

It should be stressed that the presented above approach of treating first $n$
elements of the sequence $\left\{ p_{j}\left( x\right) \right\} _{j\geq 0}$
as a $(n+1)-$vector being the result of multiplication of matrix $\mathbf{%
\Pi }_{n}$ by vector $\mathbf{X}_{n}$ is very fruitful although not
original. Traces of it appear in \cite{Bre80} or \cite{Chih79} and can be
found even earlier. Its relation to Choleski decomposition of the moment
matrix and its inverse are also not original. However such view appears in
the literature as 'yet another possibility' of looking on the main result.
In this paper this approach is a basic tool to get known results and new
ones mostly concerning connection and linearization formulae.

Most of our results concern the so called "truncated moment problem" that is
we, in fact, assume that we know finite number (say $2n+1$ including moment
or order $0)$ of moments of some distribution. That is, in many cases, the
assumption that the matrix $\mathbf{M}_{n}$ exists for all $n$ will not be
needed. Then, without using Gram--Schmidt procedure, we derive $n+1$
polynomials $\{p_{i}\}_{i=0}^{n}$ that are mutually orthogonal, we find
coefficients of expansion of $x^{i}$ in terms of these polynomials as well
as we derive all of the so called linearization coefficients i.e.
coefficients of the expansions $p_{i}(x)p_{j}(x)$ in polynomials $%
\{p_{i}\}_{i=0}^{i+j}$ for all $i,j\geq 0.$

Given two distributions (say $\alpha $ and $\delta )$ and two respective
moment sequences we are able to derive the so called "connection
coefficients" i.e. coefficients of the expansion of say $p_{j}\left(
x,\delta \right) $ in $\{p_{i}(x,\alpha )\}_{i=0}^{j}$ and conversely.

Due to very efficient numerical algorithms of Cholesky decomposition and
inversion of lower triangular matrices all these calculations can be done
within seconds using today's computers.

Of course we present also results that require existence of all moments.
These are some limit properties of arithmetic averages of orthogonal
polynomials and more importantly results concerning expansions of
Radon--Nikodym derivatives of one distribution with respect to the other
(see (\ref{rozkl})).

The paper is organized as follows. In the next Section \ref{chol} we present
consequences of our approach and derive some known and unknown results of
lesser importance. We do this basically to illustrate the usefulness of our
approach. By the end of this Section in Subsection \ref{recurr} we relate
coefficients of the power series expansion of polynomials $\left\{
p_{n}\right\} _{n\geq 0}$ to the coefficients of the 3-term recurrence
satisfied by these polynomials. More precisely we partially solve systems of
equations (\ref{_f}), (\ref{_s}) and (\ref{lam1}), (\ref{lam2}), (\ref{lam3}%
). The solution is exact and complete (see (\ref{sol11}), (\ref{sol21}) and (%
\ref{sol2})) in the case of symmetric measures $\alpha $.

We think that particularly interesting and new are results concerning
connection coefficients presented in Section \ref{con-coef}. We present
there not only formula for the connection coefficients between the two sets
of orthogonal polynomials related to two different measures but also
expansion of the Radon--Nikodym derivative of one measure with respect to
the other in a Fourier series of orthogonal polynomials related to one of
the measures. We give there two nontrivial examples. Interesting and new
seems also Section \ref{Line} presenting general formula for the
linearization coefficients. Longer and uninteresting proofs are shifted to
Section \ref{dow}.

\section{Cholesky decomposition and its consequences\label{chol}}

Our basic tool in what follows is the so called Cholesky decomposition of
the symmetric, positive definite matrix. Below we collect some of the
properties of the Cholesky decomposition in the following simple proposition.

\begin{proposition}
\label{Cholesky}Suppose a positive, normalized measure $\alpha $ has support
of infinite cardinality and $\int x^{2N}d\alpha <\infty $ for some $N\geq 0$%
. Then

i) there exists unique real, non-singular lower triangular matrix $\mathbf{L}%
_{N}(\alpha )$ such that $\mathbf{M}_{N}(\alpha )\allowbreak =\allowbreak 
\mathbf{L}_{N}(\alpha )\mathbf{L}_{N}^{T}(\alpha ),$

ii) entries of matrix $\mathbf{L}_{N}$ can be calculated recursively 
\begin{equation}
l_{n,n}=\sqrt{m_{2n}-\sum_{j=0}^{n-1}l_{n,j}^{2}},~~l_{n+1,k}=(m_{n+k+1}-%
\sum_{j=0}^{k-1}l_{n+1,j}l_{k,j})/l_{k,k},  \label{gen}
\end{equation}%
with $l_{0,0}\allowbreak =\allowbreak 1$ for $n\allowbreak =\allowbreak
0,\ldots ,N.$ Entries $l_{n,n}$ have the following interpretation:%
\begin{equation}
l_{n,n}^{2}\allowbreak =\allowbreak \frac{\Delta _{n}}{\Delta _{n-1}},
\label{l2n}
\end{equation}%
where the sequence $\left\{ \Delta _{n}\right\} $ is defined by (\ref{det}).
In particular we have: 
\begin{eqnarray}
l_{1,1}\allowbreak &=&\allowbreak \sqrt{m_{2}-m_{1}^{2}},~~l_{2,2}=\sqrt{%
m_{4}\allowbreak -\allowbreak m_{2}^{2}-\frac{(m_{3}-m_{1}m_{2})^{2}}{%
(m_{2}-m_{1}^{2})}},  \label{_odch} \\
l_{i,0}\allowbreak &=&\allowbreak m_{i},~~l_{i,1}\allowbreak =\allowbreak 
\frac{(m_{i+1}-m_{i}m_{1})}{l_{1,1}},  \label{mom} \\
l_{i,2}\allowbreak &=&\allowbreak \frac{1}{l_{2,2}}(m_{i+2}-m_{i}m_{2}%
\allowbreak -\allowbreak \frac{(m_{3}-m_{2}m_{1})(m_{i+1}-m_{i}m_{1})}{%
(m_{2}-m_{1}^{2})}),  \label{li2}
\end{eqnarray}%
$i\allowbreak =\allowbreak 1,\ldots ,N,\allowbreak $

iii) $\forall ~0\leq i,j\leq N$%
\begin{equation*}
m_{i+j}\allowbreak =\allowbreak \sum_{k=0}^{\min (i,j)}l_{i,k}l_{j,k}.
\end{equation*}
\end{proposition}

\begin{proof}
i) Follows the existence and uniqueness of the Cholesky decomposition (see
e.g. Theorem 8.2.1 of \cite{Serre}) and the fact that if the support of a
positive measure is infinite then matrix $\mathbf{M}_{N}$ exists, is
symmetric and positive definite. Besides by the Cauchy Theorem we have $%
\Delta _{n}\allowbreak =\allowbreak \left( \det L_{n}\right) ^{2}\allowbreak
=\allowbreak \prod_{i=0}^{n}l_{i,i}^{2}.$ Since $\Delta _{n-1}\allowbreak
=\allowbreak \prod_{j=0}^{n-1}l_{i,i}^{2}$ we get our assertion. ii) Follows
one of the algorithms of obtaining Cholesky decomposition (so called
Cholesky--Banachiewicz algorithm that can be found in. e.g. \cite{Fox64}).
\end{proof}

Some, more or less obvious, consequences and observations we collect in the
next proposition:

\begin{proposition}
\label{interpretacja}Let $\alpha $ be a certain measure. Let $\mathbf{M}%
_{n}(\alpha )$ and $\mathbf{M}_{n}^{-1}(\alpha ),$ be respectively $n-$th
moment matrix and its inverse. Denote $\mathbf{M}_{n}^{-1}\allowbreak
=\allowbreak \lbrack \mu _{i,j}^{(n)}]_{0\leq i,j\leq n}$. Let $\mathbf{L}%
_{n}\mathbf{L}_{n}^{T}\allowbreak =\allowbreak \mathbf{M}_{n}$ be the
Cholesky decomposition of the matrix $\mathbf{M}_{n}$, then

i) $\forall n\geq 0,$ $\mathbf{\Pi }_{n}=\mathbf{L}_{n}^{-1},$ $\mathbf{%
\Lambda }_{n}=\mathbf{L}_{n}.$ That is $\mathbf{\Lambda }_{n}\mathbf{\Lambda 
}_{n}^{T}\allowbreak =\allowbreak \mathbf{M}_{n}$ and $\mathbf{\Pi }_{n}^{T}%
\mathbf{\Pi }_{n}\allowbreak =\allowbreak \mathbf{M}_{n}^{-1}$ in particular 
$\sum_{k=\max (i,j)}^{n}\pi _{k,i}\pi _{k,j}\allowbreak =\allowbreak \mu
_{i,j}^{(n)}$ and $\sum_{k=0}^{\min (i,j)}\lambda _{i,k}\lambda
_{j,k}\allowbreak =\allowbreak m_{i+j}.$

ii) $\mathbf{P}_{n}^{T}(x)\mathbf{P}_{n}(y)\allowbreak =\allowbreak
\sum_{i=0}^{n}p_{i}(x)p_{i}(y)\allowbreak =\allowbreak \mathbf{X}_{n}^{T}%
\mathbf{M}_{n}^{-1}\mathbf{Y}_{n}$, thus $\mathbf{X}_{n}^{T}\mathbf{M}%
_{n}^{-1}\mathbf{Y}_{n}$ is the reproducing kernel and $1/\mathbf{X}_{n}^{T}%
\mathbf{M}_{n}^{-1}\mathbf{X}_{n}$ is the Christoffel function of the
measure $\alpha .$ Consequently 
\begin{eqnarray*}
\left\vert \mathbf{X}_{n}^{T}\mathbf{M}_{n}^{-1}\mathbf{Y}_{n}\right\vert
&\leq &\frac{1}{\xi _{0,n}}\sqrt{(1+\ldots +x^{2n})(1+\ldots +y^{2n})}, \\
\frac{\xi _{0,n}}{1+\ldots +x^{2n}}\, &\leq &\frac{1}{\mathbf{X}%
_{n}^{T}M_{n}^{-1}\mathbf{X}_{n}}\leq \frac{\xi _{n,n}}{1+\ldots +x^{2n}},
\end{eqnarray*}%
where $\xi _{0,n}\leq \xi _{1,n}\leq \ldots \leq \xi _{j,n}\leq \ldots \leq
\xi _{n,n}$ denote eigenvalues of the matrix $\mathbf{M}_{n}$ in
non-decreasing order.$.$

iii) $\int \mathbf{P}_{n}^{T}(x)\mathbf{M}_{n}\mathbf{P}_{n}(x)\allowbreak
d\alpha (x)=\allowbreak \sum_{i=0}^{n}m_{2i}\allowbreak =\allowbreak \xi
_{0,n}+\ldots +\xi _{n,n}.$

iv) $\frac{1}{2\pi }\int_{0}^{2\pi }\mathbf{P}_{n}(e^{it})\mathbf{P}%
_{n}^{T}(e^{-it})dt\allowbreak =\allowbreak \mathbf{\Pi }_{n}\mathbf{\Pi }%
_{n}^{T}$, consequently $\frac{1}{2\pi }\sum_{j=0}^{n}\int_{0}^{2\pi
}\left\vert p_{j}(e^{it})\right\vert ^{2}dt\allowbreak =\allowbreak \func{tr}%
(\mathbf{M}_{n}^{-1})\allowbreak =\allowbreak \sum_{j=0}^{n}1/\xi _{j,n},$

v) $\frac{1}{\xi _{n,n}}\leq \sum_{j\geq 0}^{n}\left\vert
p_{j}(0)\right\vert ^{2}\allowbreak =\allowbreak \mu _{0,0}^{(n)}\allowbreak
\leq \allowbreak \frac{1}{\xi _{0,n}},$

vi) $\frac{1}{\sqrt{n+1}\log ^{2}(n+2)}\sum_{i=0}^{n}p_{i}(x,\alpha
)\longrightarrow 0,$ $\alpha -$a.s. as $n\longrightarrow \infty .$
\end{proposition}

\begin{proof}
Is shifted to Section \ref{dow}.
\end{proof}

\begin{remark}
Part of assertion i) namely the statement $\mathbf{\Pi }_{n}^{T}\mathbf{\Pi }%
_{n}\allowbreak =\allowbreak \mathbf{M}_{n}^{-1}$ and assertion iv) were
shown in \cite{Berg11}. We presented these statements for the completeness
of the paper.
\end{remark}

\begin{remark}
Assertion vii) of Proposition \ref{interpretacja} gives in fact an estimate
of the speed of convergence in Law of Large Numbers that sequence of
orthogonal polynomials satisfies. Namely this assertion can be written in
the form $\frac{\sqrt{n+1}}{\log ^{2}(n+2)}\frac{1}{n+1}%
\sum_{i=0}^{n}p_{i}(x,\alpha )\longrightarrow 0,\alpha -a.s.$as $%
n\longrightarrow \infty .$ This result is in the spirit of \cite{Morg55} and
his followers.
\end{remark}

As a corollary we have the following observations:

\begin{corollary}
Coefficients $a_{n}$ and $b_{n}:$ $n\geq 0$ defining the 3-term recurrence
are related to the moment matrix by the formulae:%
\begin{equation}
a_{n}^{2}\allowbreak =\allowbreak \frac{\Delta _{n}\Delta _{n-2}}{\Delta
_{n-1}^{2}},~\text{~}b_{n}\allowbreak =\allowbreak \frac{\Delta _{n-1}}{%
\Delta _{n}}l_{n+1,n}l_{n,n}-\frac{\Delta _{n-2}}{\Delta _{n-1}}%
l_{n,n-1}l_{n-1,n-1},  \label{wsp}
\end{equation}%
for $n\geq 2$ with $a_{0}=0,$ $a_{1}^{2}\allowbreak =\allowbreak \Delta
_{2}\allowbreak =\allowbreak m_{2}-m_{1}^{2}.$
\end{corollary}

\begin{proof}
Following (\ref{coeff}) and Proposition \ref{interpretacja} i) we deduce $%
a_{n}^{2}\allowbreak =\allowbreak \pi _{n-1,n-1}^{2}/\pi _{n,n}^{2}.$ Since $%
\pi _{n,n}\allowbreak =\allowbreak l_{n,n}^{-1}$ we apply (\ref{l2n}). To
get formula for $b_{n}$ first we observe that $(i,i-1)$ entry of the the
inverse of the lower triangular matrix $\mathbf{L}_{n}\allowbreak
=\allowbreak \lbrack l_{i,j}]_{i=0,\ldots ,n,j=0,\ldots ,i}$ is equal to $-%
\frac{l_{i,i-1}}{l_{i,i}l_{i-1,i-1}}.$ Besides dividing both sides of (\ref%
{_s}) with $j\allowbreak =\allowbreak n$ by $\pi _{n,n}$ we get:%
\begin{equation*}
\frac{\pi _{n+1,n}}{\pi _{n+1,n+1}}+b_{n}=\frac{\pi _{n,n-1}}{\pi _{n,n}}.
\end{equation*}%
Now we have $\frac{\pi _{n+1,n}}{\pi _{n+1,n+1}}\allowbreak =\allowbreak -%
\frac{l_{n+1,n}l_{n+1,n+1}}{_{l_{n+1,n+1}l_{n,n}}}\allowbreak =\allowbreak -%
\frac{l_{n+1,n}}{l_{n,n}}$ and similarly $\frac{\pi _{n,n-1}}{\pi _{n,n}}%
\allowbreak =\allowbreak -\frac{l_{n,n-1}}{l_{n-1,n-1}}.$ Finally we use (%
\ref{l2n}).
\end{proof}

\subsection{Coefficients of the 3-term recurrence\label{recurr}}

In this subsection we will formulate a sequence of observations concerning
the two systems of equations (\ref{_f})-(\ref{lam3}) which relate
coefficients of the 3-term recurrence to the coefficients $\pi _{n,i}$ and $%
\lambda _{n,j}.$

\begin{proposition}
\label{3t-r}i) $\forall n\geq 0:a_{n}>0.$

ii) Let us denote $\eta _{n,i}\allowbreak =\allowbreak \pi
_{n,i}\prod_{j=1}^{n}a_{j}$, $\tau _{n,i}\allowbreak =\allowbreak \lambda
_{n,i}/\prod_{j=1}^{i}a_{j}\allowbreak =\allowbreak \lambda _{n,i}/\lambda
_{i,i}$ and by $\tilde{p}_{n}(x)$ let us denote the monic version of a
polynomial $p_{n}(x)$ then%
\begin{eqnarray*}
\tilde{p}_{n}(x) &=&\sum_{k=0}^{n}\eta _{n,k}x^{k}, \\
x^{n}\allowbreak &=&\allowbreak \sum_{k=0}^{n}\tau _{n,k}\tilde{p}_{k}(x).
\end{eqnarray*}

Coefficients $\left\{ \eta _{n,j},\tau _{n,j}\right\} _{n\geq 0,0\leq j\leq
n}$ satisfy the following system of equations:%
\begin{eqnarray}
\eta _{n+1,0} &=&-b_{n}\eta _{n,0}\allowbreak -a_{n}^{2}\eta _{n-1,0},
\label{_1} \\
\eta _{n+1,j} &=&\eta _{n,j-1}-b_{n}\eta _{n,j}-a_{n}^{2}\eta _{n-1,j},
\label{_2} \\
\tau _{n+1,0}\allowbreak &=&\allowbreak b_{0}\tau _{n,0}+\tau
_{n,1}a_{1}^{2},  \label{_3} \\
\tau _{n+1,j} &=&\tau _{n,j-1}+b_{j}\tau _{n,j}+a_{j+1}^{2}\tau _{n,j+1},
\label{_4}
\end{eqnarray}%
$n\geq 0,$ $j\leq n,$ with $\eta _{0,0}\allowbreak =\allowbreak \eta
_{n,n}\allowbreak =\allowbreak 1,$ and $\tau _{0,0}\allowbreak =\allowbreak
\tau _{n,n}\allowbreak =\allowbreak 1$ for $n\geq 0$.

iii) $\forall i>j:\sum_{k=i}^{j}\eta _{j,k}\tau _{k,i}\allowbreak
=\allowbreak 0\allowbreak =\allowbreak \sum_{k=i}^{j}\tau _{j,k}\eta _{k,i}$.
\end{proposition}

\begin{proof}
Is shifted to Section \ref{dow}.
\end{proof}

\begin{proposition}
\label{aux}Let us consider $4$ auxiliary number sequences $\left\{ \xi
_{n,j}^{\left( i\right) }\right\} _{n,j\geq 0},$ $\left\{ \zeta
_{n,j}^{\left( i\right) }\right\} _{n,j\geq 0}$, $i\allowbreak =\allowbreak
1,2$ satisfying the following systems of recurrences for $n\geq 0$, 
\begin{eqnarray}
\xi _{n+1,0}^{(1)}\allowbreak &=&\allowbreak -a_{n}^{2}\xi
_{n-1,0}^{(1)},~\xi _{n+1,0}^{(2)}=-b_{n}\xi _{n,0}^{(2)},  \label{_11} \\
\xi _{n+1,j}^{(1)}\allowbreak &=&\allowbreak \xi _{n,j-1}^{(1)}-a_{n}^{2}\xi
_{n-1,j}^{(1)},~\xi _{n+1,j}^{(2)}=\xi _{n,j-1}^{(2)}-b_{n}\xi _{n,j}^{(2)},
\label{_22} \\
\zeta _{n+1,0}^{(1)}\allowbreak &=&\allowbreak a_{1}^{2}\zeta
_{n,1}^{(1)},~\zeta _{n+1,0}^{(2)}=b_{0}\zeta _{n,0}^{(2)},  \label{_33} \\
\zeta _{n+1,j}^{(1)}\allowbreak &=&\allowbreak \zeta
_{n,j-1}^{(1)}+a_{j+1}^{2}\zeta _{n,j+1}^{(1)},~\zeta _{n+1,j}^{(2)}=\zeta
_{n,j-1}^{(2)}+b_{j}\zeta _{n,j}^{(2)}.  \label{_44}
\end{eqnarray}%
with $\allowbreak \allowbreak \xi _{n,j}^{(i)}\allowbreak =0$ when $%
j>n\allowbreak ,$ for $i=1,2$ and $\xi _{0,0}^{(1)}\allowbreak =\xi
_{n,n}^{(1)}\allowbreak =\allowbreak \zeta _{0,0}^{(2)}\allowbreak
=\allowbreak \zeta _{n,n}^{(2)}\allowbreak =\allowbreak 1$. Then for $j\geq
1 $ 
\begin{gather}
\xi _{n+j,n}^{(1)}=\left\{ 
\begin{array}{ccc}
0 & if & j=2k+1 \\ 
(-1)^{k}\sum_{\substack{ 1\leq j_{1}<\ldots <j_{k}\leq n+j-1  \\ %
j_{m+1}-j_{m}\geq 2,m=1,\ldots ,k-1}}\prod_{m=1}^{k}a_{j_{m}}^{2} & if & j=2k%
\end{array}%
\right. ,  \label{sol11} \\
\xi _{n+j,n}^{(2)}=(-1)^{j}\sum_{0\leq k_{1}<\ldots <k_{j}\leq
n+j-1}\prod_{m=1}^{j}b_{k_{m}},  \label{sol12} \\
\zeta _{n+l,n}^{(1)}\allowbreak =\allowbreak \left\{ 
\begin{array}{ccc}
0 & if & l=2k+1 \\ 
\sum_{j_{1}=1}^{n+1}a_{j_{1}}^{2}\sum_{j_{2}=1}^{j_{1}+1}a_{j_{1}}^{2}\ldots
\sum_{j_{k}=1}^{j_{k-1}+1}a_{j_{k}}^{2} & if & l=2k%
\end{array}%
\right. ,  \label{sol21} \\
\zeta _{n+j,n}^{(2)}\allowbreak =\allowbreak
\sum_{k_{1}=0}^{n}b_{k_{1}}\sum_{k_{2}=k_{1}}^{n}b_{k_{2}}\ldots
\sum_{k_{j}=k_{j-1}}^{n}b_{k_{j}}.  \label{sol22}
\end{gather}
\end{proposition}

\begin{proof}
Is shifted to Section \ref{dow}
\end{proof}

\begin{proposition}
\label{part} i) Let us denote $\hat{\eta}_{n,k}\allowbreak =\allowbreak \eta
_{n,k}-\xi _{n,k}^{(1)}-\xi _{n,k}^{(2)}$ and $\hat{\tau}_{n,k}\allowbreak
=\allowbreak \tau _{n,k}-\zeta _{n,k}^{(1)}-\zeta _{n,k}^{(2)},$ for $n\geq
k\geq 0.$ We have:%
\begin{eqnarray}
\hat{\eta}_{n+1,k} &=&\hat{\eta}_{n,k-1}-b_{n}\hat{\eta}_{n,k}-a_{n}^{2}\hat{%
\eta}_{n-1,k}-b_{n}\xi _{n,k}^{(1)}-a_{n}^{2}\xi _{n-1,k}^{(2)},
\label{aux1} \\
\hat{\tau}_{n+1,k} &=&\hat{\tau}_{n,k-1}+b_{k}\hat{\tau}_{n,k}+a_{k+1}^{2}%
\hat{\tau}_{n,k+1}+a_{k+1}^{2}\zeta _{n,k+1}^{(2)}+b_{k}\zeta _{n,k}^{(1)},
\label{aux2}
\end{eqnarray}%
with $\hat{\eta}_{0,0}\allowbreak =\allowbreak \hat{\eta}_{n,n}\allowbreak
=\allowbreak \allowbreak \hat{\tau}_{0,0}\allowbreak =\allowbreak \hat{\tau}%
_{n,n}\allowbreak \allowbreak =\allowbreak -1$. In particular we have:

ii) $\eta _{n+1,n}\allowbreak =\allowbreak \xi _{n+1,n}^{(2)}\allowbreak
=\allowbreak -\tau _{n+1,n}$ for $n\geq 0,$

iii) $\eta _{n+2,n}\allowbreak =\allowbreak \tau _{n+2,n}\allowbreak
=\allowbreak \xi _{n+2,n}^{(2)}\allowbreak +\xi _{n+2,n}^{(1)}\allowbreak ,$
for $n\geq 0,$

iv) $\tau _{n+3,n}\allowbreak =\allowbreak \zeta _{n+3,n}^{(2)}+\zeta
_{n+2,n}^{(1)}\zeta _{n+1,n}^{(2)}+\sum_{j=1}^{n+1}a_{j}^{2}(b_{j-1}+b_{j}),$

$\eta _{n+3,n}\allowbreak =\allowbreak \xi
_{n+3,3}^{(2)}+\sum_{j=1}^{n+2}a_{j}^{2}\sum_{k=0,k\neq j,j-1}^{n+2}b_{k},$ 
\newline
$\eta _{n+4,n}\allowbreak =\allowbreak \xi _{n+4,n}^{(1)}+\xi
_{n+4,n}^{(2)}+\sum_{k=1}^{n+3}a_{i}^{2}\sum_{0\leq i<j\leq n+3,i,j\neq
k.k-1}b_{i}b_{j},$ \newline
$\tau _{n+4,n}\allowbreak =\allowbreak -\eta _{n+4,n}-\eta _{n+4,n+1}\tau
_{n+1,n}-\eta _{n+4,n+2}\tau _{n+2,n}-\eta _{n+4,n+3}\tau _{n+3,n}.$

Assume that $\forall n\geq 0:b_{i}\allowbreak =\allowbreak 0,$ then:

v) $\eta _{n,0}\allowbreak =\allowbreak \left\{ 
\begin{array}{ccc}
0 & if & n=2k-1 \\ 
(-1)^{k}\prod_{j=1}^{k}a_{2j-1}^{2} & if & n=2k%
\end{array}%
\right. ,$ $k\allowbreak =\allowbreak 1,2,\ldots .$

\begin{equation}
\eta _{n+l,n}\allowbreak =\xi _{n+l,n}^{(1)},\tau _{n+l,n}\allowbreak
=\allowbreak \zeta _{n+l,n}^{(1)},  \label{sol2}
\end{equation}

for $l,\allowbreak n\geq 0.$
\end{proposition}

\begin{proof}
Is shifted to Section \ref{dow}
\end{proof}

\begin{remark}
As pointed out in Proposition \ref{3t-r},ii) coefficients $\eta _{i,j}$ are
the power coefficients of monic orthogonal polynomials i.e. orthogonal
polynomials with the leading coefficient equal to $1.$ The similar formula
to (\ref{sol2}) for orthogonal polynomials on the unit circle was proved by
in \cite{Gol07}. Formulae given in assertion ii) and iii) were given in \cite%
{Chih79} (Thm. 4.2 (d) and ibidem Exercise 4.1, p. 24). We present them for
completeness of the paper.
\end{remark}

\begin{remark}
Notice that $(-1)^{k}\sum_{\substack{ 1\leq j_{1}<\ldots <j_{k}\leq n-1  \\ %
j_{m+1}-j_{m}\geq 2,m=1,\ldots ,k-1}}\prod_{m=1}^{k}a_{j_{m}}^{2}$ can also
be written as 
\begin{equation*}
(-1)^{k}\sum_{j_{1}=1}^{n-2k+1}a_{j_{1}}^{2}%
\sum_{j_{2}=j_{1}+2}^{n-2k+3}a_{j_{2}}^{2}\ldots
\sum_{j_{k}=j_{k-1}+2}^{n-1}a_{j_{k}}^{2}.
\end{equation*}
\end{remark}

As a corollary we get the following recursive formula expressing moments in
terms of the coefficients $a_{n}$ and $b_{n}$ of the 3-term recurrence.

\begin{proposition}
\label{moments}i) $m_{j}\allowbreak =\allowbreak -\sum_{k=1}^{j-1}\eta
_{j-1,k-1}m_{k},$

If we assume that all coefficients $b_{n}\allowbreak =\allowbreak 0$ $n\geq
0,$ then we have simplified version of the previous statement:

ii) $m_{2k-1}\allowbreak =\allowbreak 0$ $k\allowbreak =\allowbreak
1,2,\ldots ,$ $m_{4}\allowbreak =\allowbreak a_{1}^{2}(a_{1}^{2}+a_{2}^{2}),$
$m_{2k}\allowbreak =\allowbreak
(\sum_{j=1}^{2k-2}a_{j}^{2})m_{2k-2}-\sum_{j=2}^{k-1}\eta
_{2k-1,2k-1-2j}m_{2k-2j},$ $k\geq 3.$
\end{proposition}

\begin{proof}
i) We use the (\ref{_x}) and (\ref{mom}) which leads to the identity $%
\forall j\geq 1$%
\begin{equation*}
\sum_{k=0}^{j}\eta _{j,k}m_{k}\allowbreak =\allowbreak 0.
\end{equation*}%
Consequently $m_{j}\allowbreak =\allowbreak -\sum_{k=0}^{j-1}\eta
_{j,k}m_{k}.$ Now we utilize (\ref{_2}) and get: 
\begin{eqnarray*}
m_{j}\allowbreak &=&\allowbreak -\sum_{k=0}^{j-1}(-b_{j-1}\eta
_{j-1,k}-a_{j-1}^{2}\eta _{j-2,k}+\eta _{j-1,k-1})m_{k} \\
&=&b_{j-1}\sum_{k=0}^{j-1}\eta _{j-1,k}m_{k}+a_{j-1}^{2}\sum_{k=0}^{j-2}\eta
_{j-2,k}m_{k}-\sum_{k=0}^{j-1}\eta _{j-1,k-1}m_{k} \\
&=&-\sum_{k=1}^{j-1}\eta _{j-1,k-1}m_{k}.
\end{eqnarray*}

ii) By i) we have $m_{2j}\allowbreak =\allowbreak \allowbreak
-\sum_{k=1}^{2j-1}\eta _{2j-1,k-1}m_{k}\allowbreak =\allowbreak
-\sum_{n=1}^{j-1}\eta _{2j-1,2n-1}m_{2n}$ since $m_{j}$ with odd $j$ are
equal to zero. Now we recall that $\eta _{2j-1,2j-3}\allowbreak =\allowbreak
-\sum_{i=1}^{2j-2}a_{i}^{2}$ by Proposition \ref{3t-r}, iv).
\end{proof}

\section{Connection coefficients and Radon--Nikodym derivatives.\label%
{con-coef}}

In this subsection we will express the so called connection coefficients
between two different sets of $N-$orthogonal polynomials. Hence let us
assume that we have two moment matrices $\mathbf{M}_{N}\left( \alpha \right) 
$ and $\mathbf{M}_{N}\left( \delta \right) .$ Let $\mathbf{L}_{N}\left(
\alpha \right) $ and $\mathbf{L}_{N}(\delta )$ be their Cholesky
decomposition matrices and $\left\{ \mathbf{P}_{N}\left( x,\alpha \right)
\right\} $ and $\left\{ \mathbf{P}_{N}(x,\delta )\right\} $ respective sets
of $N-$orthogonal polynomials. Then we have

\begin{lemma}
\label{con-c}We have 
\begin{equation*}
\mathbf{P}_{N}(x,\delta )\allowbreak =\allowbreak \mathbf{L}_{N}^{-1}(\delta
)\mathbf{L}_{N}(\alpha )\mathbf{P}_{N}\left( x,\alpha \right) ,
\end{equation*}%
or more precisely for all $n\allowbreak =\allowbreak 1,\ldots ,N:$ 
\begin{equation*}
p_{n}(x,\delta )\allowbreak =\allowbreak \sum_{k=0}^{n}\gamma _{n,k}(\delta
,\alpha )p_{k}\left( x,\alpha \right) ,
\end{equation*}%
where 
\begin{equation}
\gamma _{n,k}(\delta ,\alpha )\allowbreak =\allowbreak \sum_{j=k}^{n}\pi
_{n,j}\left( \delta \right) \lambda _{j,k}(\alpha ).  \label{conn}
\end{equation}%
Moreover, if we assume that polynomials $\left\{ \tilde{p}_{n}(x,\delta ),%
\tilde{p}_{n}(x,\alpha )\right\} _{n\geq 0}^{N}$ are assumed to be monic
then we have the same formula with coefficients $\pi $ replaced by $\eta $
and $\lambda $ by $\tau $ both defined in Proposition \ref{3t-r}.
\end{lemma}

\begin{proof}
This formula follows simple observation that 
\begin{equation*}
\mathbf{X}_{n}=\mathbf{L}_{n}(\alpha )\mathbf{P}_{n}(x,\alpha ).
\end{equation*}%
Then we apply Proposition \ref{interpretacja} i).

The fact that the same formula is satisfied by $\eta ^{\prime }s$ and $\tau
^{\prime }s$ instead by $\pi ^{\prime }s$ and $\lambda ^{\prime }s$ follows
the fact that we have $\mathbf{X}_{n}\allowbreak =\allowbreak \mathbf{\tilde{%
L}}_{n}(\alpha )\mathbf{\tilde{P}}_{n}(x,\alpha )$ where $\mathbf{\tilde{P}}%
_{n}(x,\alpha )$ denotes the vector $(1,\tilde{p}_{1}(\alpha ),\ldots ,%
\tilde{p}_{n}(\alpha ))^{T}$ while $\mathbf{\tilde{L}}_{n}(\alpha )$ denotes
lower triangular matrix with $(i,j)$ entry equal to $\tau _{i,j}(\alpha ).$
\end{proof}

\begin{corollary}
\label{f_few}Let $\left\{ b_{n}(\iota ),a_{n+1}(\iota )\right\} _{n\geq 0},$ 
$\iota \allowbreak =\allowbreak \delta ,\alpha $ denote coefficients of
3-term recurrences of polynomials respectively $\left\{ \tilde{p}%
_{n}(x,\delta )\right\} _{n\geq -1}$ and $\left\{ \tilde{p}_{n}(x,\alpha
)\right\} _{n\geq -1}.$ Then

i) 
\begin{equation*}
\gamma _{n,n-1}(\delta ,\alpha )\allowbreak =\allowbreak
\sum_{k=0}^{n-1}b_{k}(\alpha )-b_{k}(\delta ),
\end{equation*}

ii) 
\begin{gather*}
\gamma _{n,n-2}(\delta ,\alpha )\allowbreak =\allowbreak
\sum_{k=1}^{n-1}(a_{k}^{2}\left( \alpha \right) -a_{k}^{2}\left( \delta
\right) )\allowbreak +\allowbreak \frac{1}{2}(\sum_{j=0}^{n-2}(b_{j}(\alpha
)-b_{j}(\delta )))^{2}\allowbreak \\
+\allowbreak \frac{1}{2}\sum_{j=0}^{n-2}(b_{j}^{2}(\alpha )-b_{j}^{2}(\delta
))\allowbreak -\allowbreak b_{n-1}(\delta )\sum_{j=0}^{n-2}(b_{j}(\alpha
)-b_{j}(\delta )).
\end{gather*}
\end{corollary}

\begin{proof}
i) We use (\ref{conn}) with $\pi $ replaced by $\eta $ and $\lambda $
replaced by $\tau $ and get $\gamma _{n,n-1}\left( \delta ,\alpha \right)
\allowbreak =\allowbreak \eta _{n,n-1}(\delta )\tau _{n,n}\left( \alpha
\right) \allowbreak +\allowbreak \eta _{n,n}\left( \delta \right) \tau
_{n,n-1}\allowbreak =\allowbreak \sum_{k=0}^{n-1}(b_{k}(\alpha
)-b_{k}(\delta ))$ by Proposition \ref{part}, ii)

ii) $\gamma _{n,n-2}(\delta ,\alpha )\allowbreak =\allowbreak \eta
_{n,n-2}(\delta )\tau _{n-2,n-2}\left( \alpha \right) \allowbreak
+\allowbreak \eta _{n,n-1}\left( \delta \right) \tau _{n-1,n-2}\left( \alpha
\right) \allowbreak +\allowbreak \eta _{n,n}\left( \delta \right) \tau
_{n,n-2}\left( \alpha \right) \allowbreak =\allowbreak \eta _{n,n-2}(\delta
)\allowbreak +\allowbreak \tau _{n,n-2}\left( \alpha \right) \allowbreak
+\allowbreak \eta _{n,n-1}\left( \delta \right) \tau _{n-1,n-2}\left( \alpha
\right) $. Now we apply Proposition \ref{part}, iii) and do some algebra.
\end{proof}

\begin{corollary}
\label{con_symm}Let us assume that both distributions $\alpha $ and $\delta $
are symmetric and coefficients of the 3-term recurrences satisfied by monic
polynomials orthogonal with respect to distributions $\alpha $ and $\delta $
are respectively $\left\{ a_{n}\left( \alpha \right) \right\} _{n\geq 0}$
and $\left\{ a_{n}\left( \delta \right) \right\} _{n\geq 0}$ then 
\begin{equation*}
\tilde{p}_{n}\left( x,\delta \right) =\sum_{k=0}^{\left\lfloor
n/2\right\rfloor }\gamma _{n,n-2k}(\delta ,\alpha )\tilde{p}_{n-2k}(x,\alpha
),
\end{equation*}%
where 
\begin{gather*}
\gamma _{n,n-2k}\left( \delta ,\alpha \right) \allowbreak =\allowbreak \\
\sum_{m=0}^{k}(-1)^{m}\left( \sum_{\substack{ 1\leq j_{1}\leq j_{2}\ldots
,j_{m}\leq n-1,  \\ j_{k+1}-j_{k}\geq 2,  \\ k=1,\ldots ,m-1}}%
\prod_{i=1}^{m}a_{j_{1}}^{2}\left( \delta \right) \ldots a_{j_{m}}^{2}\left(
\delta \right) \right) \left( \sum_{i_{1}=1}^{k-m}a\left( \alpha \right)
_{j_{1}}^{2}\ldots \sum_{i_{k-m}=1}^{i_{k-m-1}+1}a\left( \alpha \right)
_{i_{k-m}}^{2}\right) .
\end{gather*}%
In particular 
\begin{eqnarray*}
\gamma _{n,n}\left( \delta ,\alpha \right) \allowbreak &=&\allowbreak 1, \\
\gamma _{n,0}(\delta ,\alpha )\allowbreak &=&\allowbreak \left\{ 
\begin{array}{ccc}
0 & if & n\text{ is odd} \\ 
\chi _{k} & if & n\allowbreak =\allowbreak 2k%
\end{array}%
\right. , \\
\gamma _{n,n-2}\left( \delta ,\alpha \right)
&=&\sum_{k=1}^{n-1}(a_{k}^{2}\left( \alpha \right) -a_{k}^{2}\left( \delta
\right) ),
\end{eqnarray*}%
where 
\begin{gather*}
\chi _{k}\allowbreak =\allowbreak \\
\sum_{m=0}^{k}(-1)^{m}\left( \sum_{\substack{ 1\leq j_{1}\leq j_{2}\ldots
,j_{m}\leq n-1,  \\ j_{k+1}-j_{k}\geq 2,  \\ k=1,\ldots ,m-1}}%
\prod_{i=1}^{m}a\left( \delta \right) _{j_{1}}^{2}\ldots a\left( \delta
\right) _{j_{m}}^{2}\right) \left( \sum_{i_{1}=1}^{k-m}a\left( \alpha
\right) _{j_{1}}^{2}\ldots \sum_{i_{k-m}=1}^{i_{k-m-1}+1}a\left( \alpha
\right) _{i_{k-m}}^{2}\right) .
\end{gather*}
\end{corollary}

\begin{remark}
It turns out that pairs of systems of orthogonal polynomials with the
property that the connection coefficients between them are nonnegative are
important. Basing on Corollaries \ref{f_few} and \ref{con_symm} we see that
a necessary conditions for coefficients $\gamma _{n,j}\left( \delta ,\alpha
\right) $ to be nonnegative is that $\forall n\geq
0:\sum_{j=0}^{n}b_{j}\left( \alpha \right) \geq \sum_{j=0}^{n}b_{j}\left(
\alpha )\right) .$ If the measures that orthogonalize those systems of
polynomials are such that $\forall n\geq 0:b_{n}\left( \delta \right)
\allowbreak =\allowbreak b_{n}\left( \alpha \right) $ then the necessary
condition for the coefficients $\gamma _{n,j}\left( \delta ,\alpha \right) $
to be nonnegative is $\forall n\geq 0:\sum_{j=1}^{n}a_{j}^{2}\left( \alpha
\right) \geq \sum_{j=1}^{n}a_{j}^{2}\left( \delta \right) .$ The discussion
why the non-negativity of connection coefficients is important and what are
the consequences of this fact is given in \cite{Szwarc92}.
\end{remark}

Following slight modification (ratio of densities is substituted by the
Radon--Nikodym derivative of respective measures) of Proposition 1 iii) of 
\cite{Szab10} (see also \cite{Szab13}) we deduce the following general
statement concerning :

\begin{corollary}
If $\frac{d\alpha }{d\delta }(x)\allowbreak =\allowbreak 1/Q_{r}(x)$ where $%
Q_{r}$ is a polynomial of order $r$ (positive on $\limfunc{supp}\delta )$
then for $N\geq r+1$ the symmetric matrix 
\begin{equation*}
\mathbf{L}_{N}^{-1}(\alpha )\mathbf{M}_{N}(\delta )\left( \mathbf{L}%
_{N}^{-1}(\alpha )\right) ^{T}
\end{equation*}
is a '$r-$ribbon' matrix i.e. its $(i,j)$ entries such that $\left\vert
i-j\right\vert >r$ are zeros.
\end{corollary}

\begin{proof}
Using the above mentioned Proposition we deduce that the lower triangular
matrix $\mathbf{L}_{N}^{-1}(\alpha )\mathbf{L}_{N}(\delta )$ is a '$r-$%
ribbon' matrix. We have $\mathbf{L}_{N}^{-1}(\alpha )\mathbf{L}_{N}(\delta )(%
\mathbf{L}_{N}^{-1}(\alpha )\mathbf{L}_{N}(\delta ))^{T}\allowbreak
=\allowbreak \mathbf{L}_{N}^{-1}(\alpha )\mathbf{M}_{N}(\delta )\left( 
\mathbf{L}_{N}^{-1}(\alpha )\right) ^{T}$ . Then we use the fact that if $%
\mathbf{A}$ is a lower triangular '$r-$ribbon' matrix then $\mathbf{AA}^{T}$
is also a '$r-$ribbon' matrix.
\end{proof}

As a more interesting consequence of Lemma \ref{con-c} we have an important
expansion of the Radon--Nikodym derivative of two measures $\alpha <<\delta
. $

\begin{theorem}
\label{expansion}Let the two measures $\alpha $ and $\delta $ both having
all moments be such that $\alpha <<\delta $ and $\int (\frac{d\alpha }{%
d\delta }(x))^{2}d\delta (x)\allowbreak <\allowbreak \infty ,$ where $\frac{%
d\alpha }{d\delta }(x)$ denotes their Radon--Nikodym derivative. Then%
\begin{equation}
\frac{d\alpha }{d\delta }(x)\allowbreak =\allowbreak \sum_{j=0}^{\infty }%
\mathbb{E}_{\alpha }p_{j}(Z,\delta )p_{j}(x,\delta ),  \label{rozkl}
\end{equation}%
in $L_{2}(\limfunc{supp}\delta ,\mathcal{F},d\delta ),$ where $\mathcal{F}$
denotes Borel sigma field of $\limfunc{supp}\delta .$ In particular we have
(Parseval's formula) 
\begin{equation}
\int (\frac{d\alpha }{d\delta }(x))^{2}d\delta (x)=\sum_{j\geq 0}(\mathbb{E}%
_{\alpha }p_{j}(Z,\delta ))^{2}.  \label{bessel}
\end{equation}%
Additionally when $\sum_{j\geq 0}(\mathbb{E}_{a}p_{j}(Z,\delta ))^{2}\ln
(j+1)^{2}\allowbreak <\allowbreak \infty ,$ we have $\delta $ almost
everywhere convergence.
\end{theorem}

\begin{proof}
Although the idea of this simple in fact theorem appeared in \cite{Szab10}
where also its numerous nontrivial applications were presented we will give
its simple proof for completeness of the paper. \newline
Radon--Nikodym derivative $\frac{d\alpha }{d\delta }(x)$ is square
integrable with respect to the measure $\delta $ i.e. hence it can be
expanded in a Fourier series with respect to the system of orthogonal
polynomials $\left\{ p_{j}\left( x,\delta \right) \right\} _{j\geq 0}$%
\begin{equation*}
\frac{d\alpha }{d\delta }(x)=\sum_{j\geq 0}\omega _{j}p_{j}\left( x,\delta
\right) .
\end{equation*}%
Now let us multiply both sides of this expansion by $p_{k}\left( x,\delta
\right) $ and integrate with respect to $\delta \left( dx\right) .$ On the
right hand side we will get $\omega _{k}$ while on the left hand side $\int
p_{k}\left( x,\delta \right) \alpha \left( dx\right) \allowbreak
=\allowbreak \mathbb{E}_{\alpha }p_{k}\left( Z,\delta \right) .$ (\ref%
{bessel}) follows Bessel equality of orthogonal series. If $\sum_{j\geq 0}(%
\mathbb{E}_{a}p_{j}(Z,\delta ))^{2}\ln (j+1)^{2}\allowbreak <\allowbreak
\infty $ then we apply Rademacher--Menshov Theorem (see e.g. \cite{Alexits})
and get almost everywhere convergence.
\end{proof}

\begin{remark}
Let us notice that if we write $p_{n}(x,\delta )\allowbreak =\allowbreak
\sum_{i=0}^{n}\gamma _{n,i}(\delta ,\alpha )p_{i}(x,\alpha ),$ then $\gamma
_{n,0}(\delta ,\alpha )\allowbreak =\allowbreak \mathbb{E}_{\alpha
}p_{n}(Z,\delta )$ after integrating both sides with respect to $\alpha
\left( dx\right) .$
\end{remark}

\begin{example}
As a corollary we will get the famous Poisson--Mehler expansion formula ((%
\ref{P-M}), below). In order not to repeat too many known details we refer
the reader to \cite{Szab10}, \cite{SzablAW} as far as the ideas and
calculations are concerned and to \cite{IA} in order to get more properties
of the mentioned below families of orthogonal polynomials.

Namely we will consider the so called $q-$Hermite polynomials defined for $%
\left\vert q\right\vert <1$ as $H_{n}(x|q)/\sqrt{[n]_{q}!},$ where $%
H_{n}(x|q)$ are monic polynomials satisfying 3-term recurrence given by
(2.3) of \cite{SzablAW}.

We used here traditional notation common in the so called $q-$series theory: 
$[n]_{q}\allowbreak =\allowbreak (1-q^{n})/(1-q),$ for $\left\vert
q\right\vert <1$ and $[n]_{1}\allowbreak =\allowbreak n,$ $%
[n]_{q}!\allowbreak =\allowbreak \prod_{j=1}^{n}[j]_{q},$ with $[0]_{q}!=1$ $%
(a)_{n}\allowbreak =\allowbreak \prod_{i=0}^{n-1}(1-aq^{i}),$ (the so called 
$q-$Pochhammer symbol).

One can consider also the case $q\allowbreak =\allowbreak 1$ obtaining
similar results but for the sake of simplicity let us consider only the case 
$\left\vert q\right\vert <1.$ Let us mention only that for $q\allowbreak
=\allowbreak 1,$ $q-$Hermite polynomials are in fact equal to the classical
Hermite polynomials, more precisely the ones that are orthogonal with
respect to measure with the density $\exp \left( -x^{2}/2\right) /\sqrt{2\pi 
}.$

It is known that $q-$Hermite polynomials are orthogonal for $\left\vert
q\right\vert <1,$ $x\in S\left( q\right) \allowbreak =\{x\in \mathbb{R}%
:\left\vert x\right\vert \leq 2/\sqrt{1-q}\}\allowbreak $ with respect to
the measure with the density $f_{N}(x|q)$ whose exact formula is not very
important and which is given e.g. in \cite{SzablAW} (formula (2.10)). The
measure with the density $f_{N}(x|q)$ is our measure $\delta .$ It is also
known (see same references) that the measure with the density $:$%
\begin{equation*}
f_{CN}\left( x|y,\rho ,q\right) =f_{N}\left( x|q\right) \prod_{k=0}^{\infty }%
\frac{(1-\rho ^{2}q^{k})}{w_{k}\left( x,y|\rho ,q\right) },
\end{equation*}%
where 
\begin{equation*}
w_{k}\left( x,y|\rho ,q\right) =(1-\rho ^{2}q^{2k})^{2}-(1-q)\rho
q^{k}(1+\rho ^{2}q^{2k})xy+(1-q)\rho ^{2}(x^{2}+y^{2})q^{2k},
\end{equation*}%
for $x,y\in S(q),$ $\left\vert \rho \right\vert <1$ for $\left\vert
q\right\vert <1$ has orthonormal polynomials equal to the so called
Al-Salam--Chihara polynomials $P_{n}(x|y,\rho ,q)$ satisfying the 3-term
recurrence given by formula (2.6) of \cite{SzablAW} divided by $\sqrt{(\rho
^{2})_{n}[n]_{q}!}$ as it follows from Proposition 1,iii) of \cite{SzablAW}
(to get orthonormality).

Measure with density the $f_{CN}$ it is our measure $\alpha .$ Following
formula (4.7) in \cite{IRS99} we deduce that 
\begin{equation*}
\mathbb{E}_{\alpha }H_{n}\left( Z|q\right) \allowbreak =\allowbreak \rho
^{n}H_{n}(y|q).
\end{equation*}%
where $y\in S(q)$ and $\left\vert \rho \right\vert <1$ are some parameters.
Details are in \cite{SzablAW} but the can be traced to earlier works of
Bryc, Matysiak, Szab\l owski \cite{bms}. 
\begin{equation*}
\frac{d\alpha }{d\delta }\left( x\right) \allowbreak =\allowbreak
\prod_{k=0}^{\infty }\frac{(1-\rho ^{2}q^{k})}{w_{k}\left( x,y|\rho
,q\right) }I_{S\left( q\right) }\left( x\right) .
\end{equation*}%
Notice also that this function is bounded from above and as such square
integrable with respect to any finite measure on $S(q).$ Again details of
the proof of this simple fact are in \cite{SzablAW}. Now following (\ref%
{rozkl}) we get:%
\begin{equation}
\prod_{k=0}^{\infty }\frac{(1-\rho ^{2}q^{k})}{w_{k}\left( x,y|\rho
,q\right) }=\sum_{j\geq 0}\frac{\rho ^{j}}{[j]_{q}!}H_{j}(x|q)H_{j}(y|q),
\label{P-M}
\end{equation}%
for every $y\in S\left( q\right) $ and almost all $x\in S\left( q\right) .$
Notice that for $q=1$ (\ref{P-M}) is also true but it requires some more
properties of Hermite polynomials.
\end{example}

\begin{remark}
Situation described above is an illustration of the situation often met in
the theory of Markov processes. Namely suppose that we have process $\mathbf{%
X}\allowbreak =\allowbreak \{X_{t}:t\in T\},$ where $T$ is some ordered set
of infinite cardinality and $\forall t\in T:$ $X_{t}$ is a random variable
with support of infinite cardinality. Suppose $dP_{t}$ is the distribution
of $X_{t}$ and that $\mathbb{E}_{t}\left\vert X_{t}\right\vert ^{n}$ is
finite for all $t$ and $n.$ Suppose also that $\left\{ p_{n}^{(t)}\right\} $
are polynomials orthogonal with respect to $dP_{t}.$ Further suppose that
the conditional distribution of $X_{s}$ given $X_{t}\allowbreak =\allowbreak
x$ for $s>t$ i.e. $dC_{s,t}$ is absolutely continuous with respect to $%
dP_{s} $ and that $\frac{dC_{s,t}}{dP_{s}}(x)$ is square integrable with
respect to $dP_{s}$ for every $s>t$ and $y\in \limfunc{supp}X_{t}.$ Then as
it follows from Theorem \ref{expansion} in $L_{2}(\limfunc{supp}X_{s},%
\mathcal{F},dP_{s})$ we have:%
\begin{equation*}
dC_{s,t}\allowbreak =\allowbreak (\sum_{j\geq 0}\mathbb{E}%
_{s,t}p_{j}^{(s)}(X_{s})p_{j}^{(s)}(x))dP_{s},
\end{equation*}%
where as usually in the theory of Markov processes $\mathbb{E}%
_{s,t}(p_{j}^{(s)}(X_{s}))$ denotes expectation with respect to distribution 
$C_{s,t}$ i.e. it denotes conditional expectation of $p_{j}^{(s)}(X_{s})$
given $X_{t}\allowbreak =\allowbreak x.$ In other words we get expansion of
the transfer function of our process.
\end{remark}

\section{Linearization coefficients\label{Line}}

Notice that Propositions \ref{Cholesky} and \ref{interpretacja} allow us to
formulate an algorithm to get so called 'linearization coefficients'. Let us
recall that linearization formula is popular name for the expansions of the
form%
\begin{equation*}
p_{n}\left( x\right) p_{m}\left( x\right) \allowbreak =\allowbreak
\sum_{j=0}^{m+n}c_{n,m,j}p_{j}(x).
\end{equation*}%
The problem is to find coefficients $c_{n,m,j}$ for all $n,$ $m\geq 1$. We
have the following lemma.

\begin{lemma}
\label{Lin_c}For $\forall n,m\geq 0$ and $s=0,\ldots ,m+n$ 
\begin{equation*}
c_{n,m,s}=\left( \sum_{\substack{ 0\leq j\leq n,  \\ 0\leq k\leq m,j+k\geq s 
}}\pi _{n,j}\pi _{m,k}\lambda _{j+k,s}\right) .
\end{equation*}
\end{lemma}

\begin{proof}
For $N\geq \max (m,n)$ we have: 
\begin{gather*}
p_{n}(x)p_{m}(x)\allowbreak =\allowbreak (\mathbf{P}_{N}(x)\mathbf{P}%
_{N}^{T}(x))_{n,m}=(\mathbf{\Pi }_{N}\mathbf{X}_{N}\mathbf{X}_{N}^{T}\mathbf{%
\Pi }_{N}^{T})_{n,m} \\
\sum_{j,k=0}^{N}\left( \mathbf{\Pi }_{N}\right) _{n,j}\left( \mathbf{X}_{N}%
\mathbf{X}_{N}^{T}\right) _{j,k}\left( \mathbf{\Pi }_{N}^{T}\right)
_{k,m}=\sum_{j,k=0}^{N}\pi _{n,j}x^{j+k}\pi _{m,k}= \\
\sum_{s=0}^{2N}p_{s}(x)\left( \sum_{j,k=0}^{N}\pi _{n,j}\pi _{m,k}\lambda
_{j+k,s}\right) .
\end{gather*}%
We now use the fact that $\pi _{n,j}\allowbreak =\allowbreak 0$ for $n<j$
and $\lambda _{k,j}\allowbreak =\allowbreak 0$ for $k<j.$
\end{proof}

\begin{remark}
Following general properties of orthogonal polynomials we deduce that $%
\forall k<|n-m|:\left( \sum_{\substack{ 0\leq i\leq n,  \\ 0\leq j\leq
m,j+i\geq k}}\pi _{n,i}\pi _{m,j}\lambda _{i+j,k}\right) \allowbreak
=\allowbreak 0.$ More precisely $c_{n,m,s}\allowbreak =\allowbreak 0$ for $%
s=0,\ldots ,\left\vert n-m\right\vert -1.$
\end{remark}

\begin{remark}
By Proposition \ref{3t-r} we deduce that for monic versions of polynomials $%
p_{n}$ we have similar formula. More precisely let $\tilde{p}_{n}(x)$ be the
monic version of polynomial $p_{n}(x)$ then 
\begin{equation*}
\tilde{p}_{n}(x)\tilde{p}_{n}\left( x\right) \allowbreak =\allowbreak
\sum_{s=0}^{n+m}\tilde{c}_{n,m,s}\tilde{p}_{s}(x),
\end{equation*}%
where 
\begin{equation}
\tilde{c}_{n,m,s}\allowbreak =\allowbreak \left( \sum_{\substack{ 0\leq
j\leq n,  \\ 0\leq k\leq m,j+k\geq s}}\eta _{n,j}\eta _{m,k}\tau
_{j+k,s}\right) .  \label{lin}
\end{equation}
This is so since $\left( \prod_{j=1}^{n}a_{j}\right) \pi _{n,k}\left(
\prod_{j=1}^{m}a_{j}\right) \pi _{m,l}\frac{1}{\prod_{j=1}^{s}a_{j}}\lambda
_{l+k,s}\allowbreak =\allowbreak \eta _{n,k}\eta _{m,l}\tau _{k+l,s}.$
\end{remark}

\begin{corollary}
We have i) 
\begin{equation*}
c_{n,m,m+m-1}\allowbreak =\allowbreak \sum_{j=\max
(n,m)}^{n+m-1}(b_{j}-b_{j-\max (n,m)}),
\end{equation*}

ii) 
\begin{gather*}
c_{n,m,n+m-2}\allowbreak =\allowbreak \sum_{j=\max
(m,n)}^{m+n-1}a_{j}^{2}\allowbreak -\allowbreak \sum_{j=1}^{\min
\{n,m)-1}a_{j}^{2}\allowbreak -\allowbreak \frac{1}{2}(\sum_{j=\max
(n,m)}^{m+n-2}b_{j}\allowbreak \\
-\allowbreak \sum_{j=0}^{\min (n,m)-1}b_{j})^{2}\allowbreak -\allowbreak 
\frac{1}{2}(\sum_{j=\max (n,m)}^{m+n-2}b_{j}^{2}\allowbreak -\allowbreak
\sum_{j=0}^{\min (n,m)-1}b_{j}^{2}).
\end{gather*}
\end{corollary}

\begin{proof}
i) By (\ref{lin}) we have $c_{n,m,n+m-1}\allowbreak =\allowbreak \eta
_{n,n}\eta _{m,m}\tau _{n+m,n\_m-1}\allowbreak +\allowbreak \eta
_{n,n-1}\eta _{m,m}\tau _{n+m-1,n+n-1}\allowbreak +\allowbreak \eta
_{n,n}\eta _{m,m-1}\tau _{n+m-1,n+n-1}\allowbreak =\allowbreak
\sum_{k=0}^{n+m-1}b_{k}\allowbreak \allowbreak -\allowbreak
\sum_{k=0}^{n-1}b_{k}\allowbreak -\allowbreak \sum_{k=0}^{m-1}b_{k}.$

ii) By (\ref{lin}) we have $c_{n,m,m+m-2}\allowbreak =\allowbreak \eta
_{n,n}\eta _{m,m}\tau _{n+m,n\_m-2}\allowbreak +\allowbreak \eta
_{n,n-2}\eta _{m,m}\tau _{n+m-2,n+m-2}+\allowbreak \eta _{n,n}\eta
_{m,m-2}\tau _{n+m-2,n+m-2}+\allowbreak \eta _{n,n-1}\eta _{m,m}\tau
_{n+m-1,n+m-2}\allowbreak +\allowbreak \eta _{n,n}\eta _{m,m-1}\tau
_{n+m-1,n+m-2}\allowbreak \allowbreak +\allowbreak \eta _{n,n-1}\eta
_{m,m-1}\tau _{n+m-2,m+n-2}\allowbreak \allowbreak =\allowbreak
\sum_{k=1}^{n+m-1}a_{k}^{2}\allowbreak -\allowbreak
\sum_{k=1}^{n-1}a_{k}^{2}-\sum_{k=1}^{m-1}a_{k}^{2}\allowbreak -\allowbreak
(\sum_{j=0}^{n-1}b_{j}\allowbreak +\allowbreak
\sum_{j=0}^{m-1}b_{j})\sum_{j=0}^{n+m-2}b_{j}\allowbreak +\allowbreak
\sum_{j=0}^{n-1}b_{j}\sum_{j=0}^{m-1}b_{j}.$ After little algebra we get the
desired form.
\end{proof}

\begin{remark}
As in the case of the connection coefficients the fact that linearization
coefficients are nonnegative is important. Why it is so, what are the
straightforward consequences of this fact and in what particular situation
it happens is again given in \cite{Szwarc92}. From the above mentioned
Corollary one can derive in fact necessary condition for linearization
coefficients to be nonnegative.
\end{remark}

\section{Proofs\label{dow}}

\begin{proof}[Proof of Proposition \protect\ref{interpretacja}]
i) Follows uniqueness of both Cholesky decomposition and orthonormal
polynomials provided sign of the leading coefficient is selected.

ii) We 
\begin{eqnarray*}
\int \mathbf{P}_{n}(x,\alpha )\mathbf{P}_{n}^{T}(x,\alpha )d\alpha
(x)\allowbreak &=&\allowbreak \mathbf{L}_{n}^{-1}\int \mathbf{X}_{n}\mathbf{X%
}_{n}^{T}d\alpha (x)\left( \mathbf{L}_{n}^{-1}\right) ^{T}\allowbreak \\
&=&\allowbreak \mathbf{L}_{n}^{-1}\mathbf{M}_{n}\left( \mathbf{L}%
_{n}^{-1}\right) ^{T}\allowbreak =\allowbreak \mathbf{I}_{n}.
\end{eqnarray*}%
Further we have 
\begin{equation*}
\mathbf{P}_{n}^{T}(x,\alpha )\mathbf{P}_{n}(y,\alpha )\allowbreak
\allowbreak =\allowbreak \mathbf{X}_{n}^{T}\allowbreak \left( \mathbf{L}%
_{n}^{-1}\right) ^{T}\mathbf{L}_{n}^{-1}\mathbf{Y}_{n}\allowbreak
=\allowbreak \left( \mathbf{X}_{n}\right) ^{T}(\mathbf{L}_{n}\mathbf{L}%
_{n}^{T})^{-1}\mathbf{Y}_{n}.
\end{equation*}%
Thus obviously we have 
\begin{equation*}
\left\vert \mathbf{X}_{n}\right\vert ^{2}/\xi _{n,n}\allowbreak \leq
\allowbreak \mathbf{X}_{n}^{T}\mathbf{M}_{n}^{-1}\mathbf{Y}_{n}\allowbreak
\leq \allowbreak \left\vert \mathbf{X}_{n}\right\vert ^{2}/\xi _{0,n},~~%
\text{and }\left\vert \mathbf{X}_{n}\right\vert ^{2}\allowbreak =\allowbreak
\sum_{i=0}^{n}x^{2i}.
\end{equation*}%
iii) 
\begin{eqnarray*}
\int \mathbf{P}_{n}^{T}(x)\mathbf{M}_{n}\mathbf{P}_{n}(x)d\alpha \allowbreak
&=&\allowbreak \int \func{tr}(\mathbf{M}_{n}\mathbf{P}_{n}(x)\mathbf{P}%
_{n}^{T}(x))d\alpha \\
&=&\allowbreak \func{tr}\mathbf{M}_{n}\mathbf{L}_{n}^{-1}\mathbf{M}%
_{n}\left( \mathbf{L}_{n}^{-1}\right) ^{T}\allowbreak =\allowbreak \func{tr}%
\mathbf{M}_{n}.
\end{eqnarray*}

iv) Denote $\mathbf{e}_{n}^{T}(t)\allowbreak =\allowbreak (1,e^{it},\ldots
,e^{int}).$ We have by Proposition \ref{interpretacja}, i) $\mathbf{P}%
_{n}(e^{it})\allowbreak =\allowbreak \mathbf{\Pi }_{n}\mathbf{e}%
_{n}^{T}(t)\allowbreak ,$ hence $\frac{1}{2\pi }\int_{0}^{2\pi }\mathbf{P}%
_{n}(e^{it})\mathbf{P}_{n}^{T}(e^{-it})dt\allowbreak \allowbreak
=\allowbreak \mathbf{\Pi }_{n}(\frac{1}{2\pi }\int_{0}^{2\pi }\mathbf{e}%
_{n}(t)\mathbf{e}_{n}^{T}(-t)dt)\mathbf{\Pi }_{n}^{T}.$ Secondly notice that 
$\left( k,j\right) -$th entry of the matrix $\mathbf{e}_{n}(t)\mathbf{e}%
_{n}^{T}(-t)$ is equal to $e^{it(k-j)}$ consequently $(\frac{1}{2\pi }%
\int_{0}^{2\pi }\mathbf{e}_{n}(t)\mathbf{e}_{n}^{T}(-t)dt)$ is equal to an
identity matrix. Second statement follows the fact that $\frac{1}{2\pi }%
\sum_{j=0}^{n}\int_{0}^{2\pi }\left\vert p_{j}(e^{it})\right\vert ^{2}dt$ is
the trace of $\frac{1}{2\pi }\int_{0}^{2\pi }\mathbf{P}_{n}(e^{it})\mathbf{P}%
_{n}^{T}(e^{-it})dt.$ But $\func{tr}(\mathbf{\Pi }_{n}\mathbf{\Pi }%
_{n}^{T})\allowbreak =\allowbreak \func{tr}(\mathbf{\Pi }_{n}^{T}\mathbf{\Pi 
}_{n})\allowbreak =\allowbreak \func{tr}\mathbf{M}_{n}^{-1}.$

v) By Proposition \ref{interpretacja}, ii) considered for $x\allowbreak
=\allowbreak y\allowbreak =\allowbreak 0.$ We get $\sum_{i=0}^{n}\left\vert
p_{i}(0)\right\vert ^{2}\allowbreak =\allowbreak \mathbf{0}_{n}^{T}\mathbf{M}%
_{n}^{-1}\mathbf{0}_{n}\mathbf{,}$ where $\mathbf{0}_{n}^{T}\allowbreak
=\allowbreak (1,0,\ldots ,0),$ which means that $\sum_{i=0}^{n}\left\vert
p_{i}(0)\right\vert ^{2}$ is $(0,0)$ entry of $\mathbf{M}_{n}^{-1}.$

vi) Let us denote $\bar{p}_{n}(x,\alpha )\allowbreak =\allowbreak \frac{1}{%
\sqrt{n+1}\log ^{2}(n+2)}\sum_{i=0}^{n}p_{i}(x,\alpha ).$ It satisfies
recursion 
\begin{equation*}
\bar{p}_{n+1}(x,\alpha )\allowbreak =\allowbreak \sqrt{\frac{n+1}{n+1}}\frac{%
\log ^{2}(n+2)}{\log ^{2}(n+3)}\bar{p}_{n}(x)+p_{n+1}(x)/\sqrt{n+2}\log
^{2}(n+3).
\end{equation*}%
Since we have $\sum_{n\geq 0}\frac{\log ^{2}(n+2)}{(n+2)\log ^{4}(n+2)}%
<\infty $ we deduce by Rademacher--Menshov theorem that series $\sum_{n\geq
0}\frac{p_{n}(x)}{\sqrt{n+1}\log ^{2}(n+2)}$ converges $\alpha -$a.s.
Further we apply \cite{Szab87} (Thm. 5).
\end{proof}

\begin{proof}[Proof of Proposition \protect\ref{3t-r}]
i) Trivial. ii) Multiplying both sides of (\ref{_f}) and (\ref{_s}) by $%
\prod_{i=1}^{n}a_{i}$ and dividing both sides of (\ref{lam1}) and (\ref{lam3}%
) by $\prod_{i=1}^{n}a_{i}$ we see that the quantities $\eta $ and $\tau $
satisfy system of equations (\ref{_1})-(\ref{_4}). iii) Follows the fact
that $j>i:$ $\sum_{k=i}^{j}\pi _{j,k}\lambda _{k,j}\allowbreak =\allowbreak
\sum_{k=i}^{j}\lambda _{j,k}\pi _{k,j}\allowbreak =\allowbreak 0$ and the
fact that $\lambda _{j,k}\pi _{k,i}\allowbreak =\allowbreak \tau _{j,k}\eta
_{k,i}$ and similarly for the product $\eta _{j,k}\tau _{k,i}.$
\end{proof}

\begin{proof}[Proof of Proposition \protect\ref{aux}]
First let us consider sequences with upper indices $\left( 1\right) .$ We
have $\xi _{n+1,0}^{(1)}\allowbreak =\allowbreak -a_{n}^{2}\xi
_{n-1,0}^{(1)} $ $.$ Recall that then $\xi _{0,0}^{(1)}\allowbreak
=\allowbreak 1$ and $\xi _{1,0}^{(1)}\allowbreak =\allowbreak 0.$ So we see
that $\xi _{n,0}^{(1)}$ with odd $n$ must be equal to zero.

To see $\xi _{n,n-2k+1}^{(1)}\allowbreak =\allowbreak 0,$ and $\zeta
_{n,n-2k-1}^{(1)}\allowbreak =\allowbreak 0,$ $k\allowbreak =\allowbreak
1,2,\ldots ,$ $n\geq 2k+1$ is easy since then our formulae (\ref{_22}) and (%
\ref{_44}) become now: 
\begin{eqnarray}
\xi _{n+1,n+1-(2k+1)}^{(1)} &=&-a_{n}^{2}\xi
_{n-1,n-1-(2k-1)}^{(1)}\allowbreak +\allowbreak \xi _{n,n-(2k+1)}^{(1)},
\label{pom} \\
\zeta _{n+1,n+1-(2k+1)}^{(1)} &=&\zeta
_{n,n-(2k+1)}^{(1)}+a_{n-1-(2k-1)}^{2}\zeta _{n-1,n-1-(2k-1).}^{(1)}
\label{pom1}
\end{eqnarray}%
We argue in case of (\ref{pom}) by induction assuming $\eta
_{n-1,n-2k}\allowbreak =\allowbreak 0$ and having $\eta _{2k+1,0}\allowbreak
=\allowbreak 0$ as shown above. In the case of (\ref{pom1}) firstly we
notice that from (\ref{_33}) with $n\allowbreak =\allowbreak 0$ we deduce
that $\zeta _{1,0}\allowbreak =\allowbreak 0$. Then taking in (\ref{pom1}) $%
n\allowbreak =\allowbreak 2$ and $k\allowbreak =\allowbreak 1$ we deduce
that $\zeta _{3,0}\allowbreak =\allowbreak 0.$ We use induction in the
similar way and deduce that $\zeta _{2k+1,0}\allowbreak =\allowbreak 0,$ $%
k\allowbreak =\allowbreak 0,\ldots $ . Now taking $k\allowbreak =\allowbreak
0$ we get 
\begin{equation*}
\zeta _{n+1,n}^{(1)}\allowbreak =\allowbreak \zeta
_{n,n-1}^{(1)}+a_{n-2}^{2}\zeta _{n-1,n}^{(1)}\allowbreak =\allowbreak \zeta
_{n,n-1}^{(1)},
\end{equation*}%
from which we deduce that $\zeta _{n,n-1}\allowbreak =\allowbreak 0$ for $%
n\geq 1.$ Now take $k\allowbreak =\allowbreak 1$ we get 
\begin{equation*}
\zeta _{n+1,n-2}^{(1)}\allowbreak =\allowbreak \zeta
_{n,n-3}^{(1)}+a_{n-2}^{2}\zeta _{n-1,n-2}^{(1)}\allowbreak =\allowbreak
\zeta _{n,n-3}^{(1)},
\end{equation*}%
from which we deduce that $\zeta _{n,n-3}\allowbreak =\allowbreak 0$ for all 
$n\geq 3.$ In the similar way we show that $\zeta _{n,n-2k-1}$ $\allowbreak
=\allowbreak 0$ for all $n\geq 2k+1.$

Hence let us consider the case of even differences in indices $i$ and $j$ in 
$\xi _{i,j}^{(1)}.$

The proofs will be by induction. Let us prove (\ref{sol11}) first. We will
prove it for indices $(n,n-2k).$ Recursive formula (\ref{_22}) becomes now:%
\begin{equation}
\xi _{n+1,n+1-2k}^{(1)}=-a_{n}^{2}\xi _{n-1,n-1-2(k-1)}^{(1)}+\xi
_{n,n-2k}^{(1)}  \label{eqn}
\end{equation}%
First notice that since sign of $\eta _{n,n-2k}$ is $(-1)^{k}$ and of $\eta
_{n-1,n-2k+1}$ is $(-1)^{k-1}$ by induction assumption we deduce that the
sign of $\eta _{n+1,n+1-2k}$ is $(-1)^{k}$ as claimed. Secondly notice that (%
\ref{sol11}) can be interpreted as a sum of products of elements of $k-$%
combinations drawn from the set $\left\{ a_{1}^{2},\ldots
,a_{n-1}^{2}\right\} $ such that distance between numbers of chosen elements
is greater than $1.$ For example from $3$ elements we can select only one
such $2-$combinations. Alter little reflection one sees that one there are $%
\binom{n-k}{k}$ such combinations consequently that $\eta _{n,n-2k}$
contains $\binom{n-k}{k}$ products. Equation (\ref{eqn}) states that sum of
such products of $k-$combinations chosen form the set $\{1,,\ldots ,n\}$ can
be decomposed on the sum of such products chosen from the set set with
indices $\left\{ 1,\ldots ,n-1\right\} $ and a sum of products containing
element $a_{n}^{2}$ times products of similarly chosen $(k-1)-$combinations
but from the set $\left\{ 1,\ldots ,n-2\right\} .$ There are $\binom{n-k}{k}$
summands of the first type and $\binom{n-1-(k-1)}{k-1}$ summand of the
second type (i.e. containing $a_{n}^{2}).$ The total number of summands in $%
\eta _{n+1,n+1-2k}$ is just%
\begin{equation*}
\binom{n-1-(k-1)}{k-1}+\binom{n-k}{k}=\binom{n+1-k}{k},
\end{equation*}%
by the well know property of the Pascal triangle as it should be.

The proof of (\ref{sol21}). Let us denote by $\beta _{n,l}$ the right hand
side of (\ref{sol2}). We have: 
\begin{eqnarray*}
\beta _{n+2k,n}-\beta _{n-1+2k,n-1}\allowbreak &=&\allowbreak
a_{n+1}^{2}\sum_{j_{2}=1}^{n+2}a_{j_{1}}^{2}\ldots \sum_{j_{k}=1}^{j_{k-1}+1}
\\
&=&a_{n+1}^{2}\beta _{n+1+2k-2,n+1}.
\end{eqnarray*}%
Further we have $\beta _{n+2,n}\allowbreak =\allowbreak
\sum_{j_{1}=1}^{n+1}a_{j_{1}}^{2}$ by direct calculation. Now notice that
sequences $\zeta ^{(1)}$ and $\beta $ satisfy the same difference equations
and have the same initial conditions. Hence they are identical.

Now let us consider sequences with upper index $(2).$ First of all notice
that by (\ref{sol22}) can be also written 
\begin{equation}
\zeta _{n+j,n}^{(2)}\allowbreak =\allowbreak
\sum_{k_{1}=0}^{n}b_{k_{1}}\sum_{k_{2}=0}^{k_{1}}b_{k_{2}}\ldots
\sum_{k_{j}=0}^{k_{j-1}}b_{k_{j}}.  \label{pp}
\end{equation}
Let us denote by $\gamma _{n+j,n}$ the left hand side of (\ref{pp}). From
this form we can easily deduce that 
\begin{equation*}
\gamma _{n+1+j,n+1}-\gamma _{n+j,n}\allowbreak =\allowbreak b_{n+1}\gamma
_{n+j,n+1}\allowbreak .
\end{equation*}%
Thus $\gamma _{n+j,n}$ satisfies the same recurrence as $\zeta
_{n+j,n}^{(2)} $ with the same initial condition. Consequently $\zeta
_{n+j,n}^{(2)}\allowbreak =\allowbreak \gamma _{n+j,n}.$

Now it remained to prove (\ref{sol12}) . First of all notice that left hand
side of (\ref{sol12}) is a sum of products of all of the sets $\left\{
b_{0},\ldots ,b_{n+j-1}\right\} $ of the size $j$. Let us denote it by $%
\delta _{n+j,n}.$ From what was stated earlier it follows that $\delta
_{n+1+j,n+1}-\delta _{n+j,n}$ is equal to the sum of product of subsets of
the set $\left\{ b_{0},\ldots ,b_{n+j-1}\right\} $ of the size $j$ that
contain element $b_{n+j-1}.$ Another word it is equal to $%
-(-1)^{j-1}b_{n+j-1}\delta _{n+j-1,n}.$ Hence $\delta _{n+j,n}$ and $\xi
_{n+j,n}^{(2)}$ satisfy the same recurrence with the same initial condition.
\end{proof}

\begin{proof}[Proof of Proposition \protect\ref{part}]
i) Combining (\ref{_1}), (\ref{_2}) with (\ref{_11}) and (\ref{_22}) we get $%
\hat{\eta}_{n+1,j}\allowbreak =\allowbreak \eta _{n+1,j}-\xi
_{n+1,j}^{(1)}-\xi _{n+1,j}^{(2)}\allowbreak =\allowbreak \eta
_{n,j-1}-b_{n}\eta _{n,j}-a_{n}^{2}\eta _{n-1,j}\allowbreak -\allowbreak
(\xi _{n,j-1}^{(1)}-a_{n}^{2}\xi _{n-1,j}^{(1)})\allowbreak -\allowbreak
(\xi _{n,j-1}^{(2)}-b_{n}\xi _{n,j}^{(2)})\allowbreak =\allowbreak \hat{\eta}%
_{n,j-1}\allowbreak -\allowbreak b_{n}(\eta _{n,j}-\xi _{n,j}^{(2)}-\xi
_{n,j}^{(1)})\allowbreak \allowbreak -\allowbreak \allowbreak b_{n}\xi
_{n,j}^{(1)}-a_{n}^{2}(\eta _{n-1,j}-\xi _{n-1,j}^{(1)}-\xi
_{n-1,j}^{(2)})\allowbreak -a_{n}^{2}\allowbreak \xi _{n-1,j}^{(2)}$ which
is first of the equations in i). Now let us consider (\ref{_3}), (\ref{_4}),
(\ref{_33}) and (\ref{_44}). We get $\allowbreak \hat{\tau}%
_{n+1,j}=\allowbreak \tau _{n+1,j}\allowbreak -\allowbreak \zeta
_{n+1,j}^{(1)}\allowbreak -\allowbreak \zeta _{n+1,j}^{(2)}\allowbreak
=\allowbreak \tau _{n,j-1}\allowbreak +\allowbreak b_{j}\tau
_{n,j}\allowbreak +\allowbreak a_{j+1}^{2}\tau _{n,j+1}\allowbreak
\allowbreak -\allowbreak (\allowbreak \zeta _{n,j-1}^{(1)}+a_{j+1}^{2}\zeta
_{n,j+1}^{(1)})\allowbreak -\allowbreak (\zeta _{n,j-1}^{(2)}+b_{j}\zeta
_{n,j}^{(2)})\allowbreak =\allowbreak \hat{\tau}_{n,j-1}\allowbreak
+\allowbreak a_{j+1}^{2}(\tau _{n,j+1}\allowbreak -\allowbreak \zeta
_{n,j+1}^{(1)}\allowbreak -\allowbreak \zeta _{n,j+1}^{(2)})\allowbreak
+\allowbreak a_{j+1}^{2}\zeta _{n,j+1}^{(2)}\allowbreak +\allowbreak
b_{j}(\tau _{n,j}\allowbreak -\allowbreak \zeta _{n,j}^{(2)}\allowbreak
-\allowbreak \zeta _{n,j}^{(1)})\allowbreak +\allowbreak b_{j}\zeta
_{n,j}^{(1)}.$

ii) Consider (\ref{aux1}) and (\ref{aux2}) with $k\allowbreak =\allowbreak
n. $ We get then $\hat{\eta}_{n+1,n}\allowbreak =\allowbreak \hat{\eta}%
_{n,n-1}$ and $\hat{\tau}_{n+1,n}\allowbreak =\allowbreak \hat{\tau}%
_{n,n-1}. $ Since for $n\allowbreak =\allowbreak 1$ we have $\hat{\eta}%
_{1,0}\allowbreak =\allowbreak \hat{\tau}_{1,0}\allowbreak =\allowbreak 0$
we get the assertion.

iii) Take $k\allowbreak =\allowbreak n-1$ in (\ref{aux1}) and (\ref{aux2}).
We get then $\hat{\eta}_{n+1,n-1}\allowbreak =\allowbreak \hat{\eta}%
_{n,n-2}\allowbreak -\allowbreak b_{n}\hat{\eta}_{n,n-1}\allowbreak
-\allowbreak a_{n}^{2}\hat{\eta}_{n-1,n-1}\allowbreak -\allowbreak b_{n}\xi
_{n,n-1}^{(1)}\allowbreak -\allowbreak a_{n}^{2}\xi
_{n-1,n-1}^{(2)}\allowbreak =\allowbreak \hat{\eta}_{n,n-2},$ since $\hat{%
\eta}_{n,n-1}\allowbreak =\allowbreak \xi _{n,n-1}^{(1)}\allowbreak
=\allowbreak 0$ and $\hat{\eta}_{n-1,n-1}\allowbreak =\allowbreak -\xi
_{n-1,n-1}^{(2)}\allowbreak =\allowbreak -1.$ Further since $\hat{\eta}%
_{2,0}\allowbreak =\allowbreak 0$ we get the assertion.

iv) As before we take $k\allowbreak =\allowbreak n-2$ in (\ref{aux1}) and (%
\ref{aux2}). We get then $\hat{\tau}_{n+1,n-2}\allowbreak =\allowbreak \hat{%
\tau}_{n,n-3}\allowbreak \allowbreak +\allowbreak b_{n-2}\hat{\tau}%
_{n,n-2}\allowbreak +\allowbreak +a_{n-1}^{2}\hat{\tau}_{n,n-1}\allowbreak
+\allowbreak a_{n-1}^{2}\zeta _{n,n-1}^{(2)}\allowbreak +\allowbreak
b_{n-2}\zeta _{n,n-2}^{(1)}\allowbreak =\allowbreak \hat{\tau}%
_{n,n-3}\allowbreak +\allowbreak \allowbreak a_{n-1}^{2}\zeta
_{n,n-1}^{(2)}\allowbreak +\allowbreak b_{n-2}\zeta _{n,n-2}^{(1)},$ since $%
\hat{\tau}_{n,n-2}\allowbreak =\allowbreak \hat{\tau}_{n,n-1}\allowbreak
=\allowbreak 0$ as shown above. Besides $\zeta _{n,n-1}^{(2)}\allowbreak
=\allowbreak \sum_{k=0}^{n-1}b_{k}$ and $\zeta _{n,n-2}^{(1)}\allowbreak
=\allowbreak \sum_{k=1}^{n-1}a_{k}^{2}$ as shown in (\ref{sol21}) and (\ref%
{sol22}). Now it a matter of algebra. We reason in the similar way in case
of $\eta _{n+3,n}$ using the fact that $\hat{\eta}_{n+2,n}\allowbreak
=\allowbreak \hat{\eta}_{n+1,n}\allowbreak =\allowbreak 0$ and knowing $\xi
_{n+2,n}^{(1)}$ and $\xi _{n+1,n}^{(2)}$ by (\ref{sol11}) and (\ref{sol12}).

So now let us consider $\eta _{n+4,n}.$ By taking $k\allowbreak =\allowbreak
n-3$ in (\ref{aux1}) we get: $\hat{\eta}_{n+1,n-3}\allowbreak =\allowbreak 
\hat{\eta}_{n,n-4}\allowbreak -\allowbreak b_{n}\hat{\eta}%
_{n,n-3}\allowbreak -\allowbreak a_{n}^{2}\hat{\eta}_{n-1,n-3}\allowbreak
-\allowbreak b_{n}\xi _{n,n-3}^{(1)}\allowbreak -\allowbreak a_{n}^{2}\xi
_{n-1,n-3}^{(2)}\allowbreak =\allowbreak \hat{\eta}_{n,n-4}\allowbreak
-\allowbreak b_{n}\hat{\eta}_{n,n-3}\allowbreak \allowbreak -\allowbreak
a_{n}^{2}\sum_{0\leq k_{1}<k_{2}\leq n-2}b_{k_{1}}b_{k_{2}}$ since $\hat{\eta%
}_{n-1,n-3}\allowbreak =\allowbreak \xi _{n,n-3}^{(1)}\allowbreak
=\allowbreak 0$ as shown above and by (\ref{sol11}). Further we use (\ref%
{sol12}) and some algebra.

v) To see that (\ref{sol2}) holds true it is enough to apply (\ref{aux1})
and (\ref{aux2}) with $b_{k}\allowbreak =\allowbreak 0,$ $k\geq 0$ which
results in $\xi _{n,k}^{(2)}\allowbreak =\allowbreak \zeta
_{n,k}^{(2)}\allowbreak =\allowbreak 0,$ for $n>0,$ $k\allowbreak \geq 0$
and which leads to relationships $\hat{\eta}_{n+1,k}=\hat{\eta}%
_{n,k-1}\allowbreak -\allowbreak a_{n}^{2}\hat{\eta}_{n-1,k}$ and $\hat{\tau}%
_{n+1,k}=\hat{\tau}_{n,k-1}\allowbreak +\allowbreak a_{k+1}^{2}\hat{\tau}%
_{n,k+1}$ with $\hat{\eta}_{n,n}\allowbreak =\allowbreak \hat{\tau}%
_{n,n}\allowbreak =\allowbreak 0,$ for $n>0$ and $\hat{\eta}%
_{i,0}\allowbreak =\allowbreak \hat{\tau}_{i,0}\allowbreak =\allowbreak 0$
for $i\allowbreak =\allowbreak 1,2.$ Now it is elementary to see that we
must have $\hat{\eta}_{n,k}\allowbreak =\allowbreak \hat{\tau}%
_{n,k}\allowbreak =\allowbreak 0$ for all $n>0,$ $k\geq 0.$
\end{proof}

\end{document}